\documentclass[11pt,reqno]{amsart}
\usepackage{fullpage}
\usepackage{mathrsfs,amssymb,graphicx,verbatim,amsmath,amsfonts}
\usepackage{paralist}
\usepackage{esint}
\usepackage{amsaddr}
\usepackage[breaklinks,pdfstartview=FitH]{hyperref}
\usepackage{upgreek}
\usepackage[mathscr]{euscript}
\usepackage{dutchcal}
\usepackage[T1]{fontenc}

\usepackage{graphicx}
\usepackage{sidecap}
\usepackage[all]{xy}
\usepackage[font=small,labelfont=bf,labelsep=endash]{caption}
\usepackage{enumitem}

\renewcommand{\eta}{\upeta}

\renewcommand{\epsilon}{\upepsilon}
\newcommand{\MM}{\mathcal{M}}

\renewcommand{\chi}{\upchi}

\renewcommand{\le}{\leqslant}
\renewcommand{\ge}{\geqslant}

\renewcommand{\setminus}{\smallsetminus}
\renewcommand{\gamma}{\upgamma}
\renewcommand{\lambda}{\uplambda}
\renewcommand{\alpha}{\upalpha}
\renewcommand{\delta}{\updelta}
\renewcommand{\beta}{\upbeta}
\renewcommand{\omega}{\upomega}
\renewcommand{\nu}{\upnu}
\renewcommand{\mu}{\upmu}
\renewcommand{\psi}{\uppsi}
\renewcommand{\phi}{\upphi}
\renewcommand{\rho}{\uprho}
\renewcommand{\kappa}{\upkappa}

\renewcommand{\tau}{\uptau}
\newcommand{\Id}{\mathsf{Id}}
\newcommand{\ud}[0]{\,\mathrm{d}}

\newcommand{\spn}{\mathrm{\bf span}}
\makeatletter
\def\moverlay{\mathpalette\mov@rlay}
\def\mov@rlay#1#2{\leavevmode\vtop{%
   \baselineskip\z@skip \lineskiplimit-\maxdimen
   \ialign{\hfil$\m@th#1##$\hfil\cr#2\crcr}}}
\newcommand{\charfusion}[3][\mathord]{
    #1{\ifx#1\mathop\vphantom{#2}\fi
        \mathpalette\mov@rlay{#2\cr#3}
      }
    \ifx#1\mathop\expandafter\displaylimits\fi}
\makeatother

\newcommand{\proj}{\mathsf{Proj}}

\newcommand{\NN}{\mathcal{N}}

\newcommand{\kk}{\mathsf{k}}

\renewcommand{\det}{\mathrm{\bf det}}

\newcommand{\sub}{\mathscr{C}}

\renewcommand{\d}{\delta}

\newcommand{\e}{\varepsilon}
\newcommand{\R}{\mathbb R}

\newcommand{\cc}{\mathsf{c}}

\newtheorem{theorem}{Theorem}
\newtheorem{lemma}[theorem]{Lemma}

\newtheorem{fact}[theorem]{Fact}

\theoremstyle{remark}
\newtheorem{remark}[theorem]{Remark}

\renewcommand{\pi}{\uppi}
\renewcommand{\zeta}{\upzeta}

\newcommand{\qs}{\mathsf{qs}}

\renewcommand{\S}{\mathsf{S}}
\renewcommand{\subset}{\subseteq}
\renewcommand{\supset}{\supseteq}
\newcommand{\C}{\mathbb C}

\DeclareMathOperator{\trace}{\bf Tr}
\newcommand{\E}{\mathbb{E}}

\renewcommand{\sigma}{\upsigma}

\newcommand{\N}{\mathbb N}
\newcommand{\Z}{\mathbb Z}

\newcommand{\eqdef}{\stackrel{\mathrm{def}}{=}}

\begin{document}

\title{Impossibility of dimension reduction  in the nuclear norm}


\author{Assaf Naor}
\address{Mathematics Department, Princeton University, Princeton, New Jersey 08544-1000, USA.}
\email{naor@math.princeton.edu}

\author{Gilles Pisier}
\address {Department of Mathematics, Texas A\&M University, College Station, Texas 77843, USA.}
\email{pisier@math.tamu.edu}

\author{Gideon Schechtman}
\address {Department of Mathematics,  Weizmann Institute of Science, Rehovot 76100, Israel.}
\email{gideon@weizmann.ac.il}

\thanks{A.~N. was supported by the BSF, the NSF, the Packard Foundation and the Simons Foundation. G.~S. was supported by the ISF. The research that is presented here was conducted under the auspices of the Simons Algorithms and Geometry (A\&G) Think Tank. A conference version of this article will appear in the proceedings of the 29th annual ACM--SIAM Symposium on Discrete Algorithms (SODA 2018).}

\date{\today}

\maketitle


\begin{abstract}
Let $\S_1$ (the Schatten--von Neumann trace class) denote the Banach space of all compact linear operators $T:\ell_2\to \ell_2$ whose nuclear norm $\|T\|_{\S_1}=\sum_{j=1}^\infty\sigma_j(T)$ is finite, where $\{\sigma_j(T)\}_{j=1}^\infty$ are the singular values of $T$.  We prove that for arbitrarily large $n\in \N$ there exists a subset  $\sub\subset \S_1$ with $|\sub|=n$ that cannot be embedded  with bi-Lipschitz distortion $O(1)$ into any $n^{o(1)}$-dimensional linear subspace of $\S_1$.  $\sub$ is not even a $O(1)$-Lipschitz quotient of any subset of any $n^{o(1)}$-dimensional linear subspace of $\S_1$. Thus, $\S_1$ does not admit a dimension reduction result \'a la Johnson and Lindenstrauss (1984), which complements the work of Harrow,  Montanaro and Short (2011) on the limitations of quantum dimension reduction under the assumption that the embedding into low dimensions is a quantum channel. Such a statement was previously known with $\S_1$ replaced by the Banach space $\ell_1$ of absolutely summable sequences  via the work of Brinkman and Charikar (2003). In fact, the above set $\sub$ can be taken to be the same set as the one that Brinkman and Charikar considered, viewed as a collection of diagonal matrices in $\S_1$. The challenge is to demonstrate that $\sub$ cannot be faithfully realized in an arbitrary low-dimensional subspace of $\S_1$, while Brinkman and Charikar obtained such an assertion only for subspaces of $\S_1$ that consist of diagonal operators (i.e., subspaces of $\ell_1$). We establish this by proving that the Markov 2-convexity constant of any finite dimensional linear subspace $X$ of $\S_1$ is at most a universal constant multiple of $\sqrt{\log \dim(X)}$.
\end{abstract}


\newpage

\section{Introduction}

The (bi-Lipschitz) distortion of a metric space $(\MM,d_\MM)$ in a metric space $(\NN,d_\NN)$, which is a numerical quantity that is commonly denoted~\cite{LLR94} by $\cc_{(\NN,d_\NN)}(\MM,d_\MM)$ or simply $\cc_\NN(\MM)$ if the metrics are clear from the context, is the infimum over those $\alpha\in [1,\infty]$ for which there exists (an embedding) $f:\MM\to \NN$ and (a scaling factor) $\lambda\in (0,\infty)$ such that
\begin{equation}\label{eq:def distortion}
\forall\, x,y\in \MM,\qquad \lambda d_\NN\big(f(x),f(y)\big) \le d_\MM(x,y)\le \alpha\lambda d_\NN\big(f(x),f(y)\big).
\end{equation}
When~\eqref{eq:def distortion} occurs one says that $(\MM,d_\MM)$ embeds with (bi-Lipschitz) distortion $\alpha$ into $(\NN,d_\NN)$.

Following~\cite{Nao17-ICM}, a Banach space $(X,\|\cdot\|_X)$ is said to {\em admit metric dimension reduction} if for every $n\in \N$, every subset $\sub\subset X$ of size $n$ embeds with distortion $O_X(1)$ into some linear subspace of $X$ of dimension $n^{o_X(1)}$. Formally, given $\alpha\in [1,\infty)$ and $n\in \N$, denote by  $\mathsf{k}_n^\alpha(X)$  the smallest $k\in \N$ such that for every $\sub\subset X$ with $|\sub|=n$ there exists a $k$-dimensional linear subspace $F=F_\sub$ of $X$ into which $\sub$ embeds with distortion $\alpha$. Using this  notation, the above terminology can be rephrased to say that $X$ admits metric dimension reduction if there exists $\alpha\in [1,\infty)$ for which
\begin{equation}\label{eq:admits reduction}
\lim_{n\to \infty}  \frac{\log \mathsf{k}_{n}^\alpha(X)}{\log n}=0.
\end{equation}
The reason why  the specific asymptotic behavior in~\eqref{eq:admits reduction} is singled out here is that, based on previous works some of which are described below, it is a recurring bottleneck in several cases of interest. Also, such behavior is what would be needed in relation\footnote{Formally, for the purpose of efficient approximate nearest neighbor search one cannot use a dimension reduction statement like~\eqref{eq:admits reduction} as a ``black box" without additional information about the low-dimensional embedding itself rather than its mere existence. One would want the embedding to be fast to compute and ``data oblivious," as in the classical Johnson--Lindenstrauss lemma~\cite{JL84}. There is no need to give a precise formulation here because the present article is devoted to ruling out any low-dimensional low-distortion embedding whatsoever.} to the existence of a nontrivial data structure for approximate nearest neighbor search in $X$, due to the  forthcoming work~\cite{ANNRW17}.

By fixing any $x_0\in \sub$ and considering $F=\spn(\sub-x_0)\subset X$ we  have $\kk_n^1(X)\le n-1$. So, the  pertinent question is to obtain a bound on $\mathsf{k}_{n}^\alpha(X)$  that is significantly smaller than $n$. This natural question turns out to be an elusive longstanding goal for all but a few of the  classical Banach spaces.

  If $X$ is a Hilbert space, then~\eqref{eq:def distortion} holds true, i.e., $\ell_2$ admits metric dimension reduction. In fact, the influential Johnson--Lindenstrauss lemma~\cite{JL84} asserts the stronger bound\footnote{We shall use throughout this article the following (standard) asymptotic notation. Given two quantities $Q,Q'>0$, the notations
$Q\lesssim Q'$ and $Q'\gtrsim Q$ mean that $Q\le \mathsf{K}Q'$ for some
universal constant $\mathsf{K}>0$. The notation $Q\asymp Q'$
stands for $(Q\lesssim Q') \wedge  (Q'\lesssim Q)$. If  we need to allow for dependence on certain parameters, we indicate this by subscripts. For example, in the presence of an auxiliary parameter $\psi$, the notation $Q\lesssim_\psi Q'$ means that $Q\le c(\psi)Q' $, where $c(\psi) >0$ is allowed to depend only on $\psi$, and similarly for the notations $Q\gtrsim_\psi Q'$ and $Q\asymp_\psi Q'$.}
\begin{equation}\label{eq:quote JL}
\forall\, \alpha\in (1,\infty),\qquad \kk_{n}^\alpha(\ell_2)\lesssim_\alpha \log n.
\end{equation}
See~\cite{JL84,Alo03,LN17} and~\cite{Art02,IN07}  for the implicit dependence on $\alpha$ in~\eqref{eq:quote JL} as $\alpha\to 1$ and $\alpha\to \infty$, respectively. By~\cite{JN09} there exists a universal constant $\alpha_0\in (0,\infty)$ and a Banach space $X$ that is not isomorphic to a Hilbert space for which $\kk_{n}^{\alpha_0}(X)\lesssim \log n$. In other words, there  exist non-Hilbertian Banach spaces that admit metric dimension reduction (even with a stronger logarithmic guarantee).

By~\cite{JLS87}  there is a universal constant $C\in (0,\infty)$ for which $\kk_{n}^\alpha(\ell_\infty)\le n^{C/\alpha}$ (see~\cite{Mat92,Mat96} for simplifications and improvements). By~\cite{Mat96} this  is sharp up to the  value of $C$ (see also~\cite{Nao17} for a stronger statement and a different proof). Therefore, $\ell_\infty$ does not admit metric dimension reduction.

The case  $X=\ell_1$  is especially important from the perspectives  of both pure mathematics and algorithms. Nevertheless, it required substantial effort to even show that, say, one has $\kk_{n}^\alpha(\ell_1)\le n/2$ for some universal constant $\alpha$: This is achieved in the forthcoming work~\cite{ANN17} which obtains the estimate $\kk_{n}^\alpha(\ell_1)\lesssim n/\alpha$. The question whether $\ell_1$ admits metric dimension reduction was open for many years, until it was resolved negatively in~\cite{BC03} by showing that there exists a universal constant $c\in (0,1)$ such that $\kk_{n}^\alpha(\ell_1)\ge n^{c/\alpha^2}$. Indeed, let $\sub \subset \ell_1$ be the finite subset that is considered in~\cite{BC03} and suppose that $F\subset \ell_1$ is a $k$-dimensional linear subspace into which $\sub$ embeds with distortion $\alpha$. By~\cite{Tal90} (the earlier estimates of~\cite{Sch87,BLM89} suffice here), $F$ embeds with distortion $O(1)$ into $\ell_1^{O(k\log k)}$.  Hence, $\sub$ embeds with distortion $O(\alpha)$ into $\ell_1^{O(k\log k)}$. This implies that $k\log k\gtrsim |\sub|^{\gamma/\alpha^2}$ for some universal constant $\gamma>0$ by the main result of~\cite{BC03}, which gives the stated lower bound on $\kk_{n}^\alpha(\ell_1)$.

 Remarkably, despite major efforts over the past three decades, the above quoted results are the entirety of what is known about metric dimension reduction in Banach spaces  in terms of the size of the point set; in particular, no nontrivial upper or lower bounds on $\kk_n^\alpha(X)$ are currently known when $X=\ell_p$ for any $p\in (1,2)\cup(2,\infty)$. The purpose of the present article  is to increase the repertoire of classical Banach spaces which fail to admit metric dimensionality reduction by one more space. Specifically, we will demonstrate that this is so for the Schatten--von Neumann trace class $\S_1$.

$\S_1$ consists of those linear operators $T:\ell_2\to \ell_2$ for which $\|T\|_{\S_1}\eqdef \sum_{j=1}^\infty\sigma_j(T)<\infty$, where $\{\sigma_j(T)\}_{j=1}$  are the singular values of $T$; see Section~\ref{sec:main} below for background (in particular, $\|\cdot\|_{\S_1}$ is a norm~\cite{vN37} which is sometimes called the {\em nuclear norm}\footnote{Those who prefer to consider the nuclear norm on $m\times m$ matrices can do so throughout, since all of our results are equivalent to their matricial counterparts; see Lemma~\ref{lem:discretization} below for a formulation of this (straightforward) statement.}). Our main result is the following theorem.

\begin{theorem}\label{thm:main intro}
There is a universal constant $c>0$ such that $\kk_{n}^\alpha(\S_1)\ge n^{c/\alpha^2}$ for all $n\in \N$ and $\alpha\ge 1$.
\end{theorem}

Since the publication of~\cite{BC03}, there was an obvious candidate for an $n$-point subset of $\S_1$ that could potentially exhibit the failure of metric dimensionality reduction in $\S_1$. Namely, since $\ell_1$ is the subspace of diagonal operators in $\S_1$, one could  consider the same subset as the one that was used  in~\cite{BC03} to rule out metric dimensionality reduction in $\ell_1$ (see Figure~\ref{fig:diamond-laakso} below). The main result of~\cite{BC03} states that this subset does not well-embed into any low-dimensional subspace of $\S_1$ all of whose elements are diagonal operators. So, the challenge amounted to strengthening this assertion so as to apply to low-dimensional subspaces of $\S_1$ whose elements can be any operator whatsoever.\footnote{One should note here that $\S_1$ does not admit a bi-Lipschitz embedding into any $L_1(\mu)$ space, as follows by combining the corresponding linear result of~\cite{McC67,Pis78} with a classical differentiation argument~\cite{BL00}, or directly by using a bi-Lipschitz invariant that is introduced in the forthcoming work~\cite{NS17-alpha}.}

This is exactly what Theorem~\ref{thm:main intro} achieves, i.e., its contribution is not a construction of a new example but rather proving that the natural guess indeed works. The analogue in $\S_1$ of the fact  that any finite-dimensional subspace of $\ell_1$ well-embeds into $\ell_1^m$ with $m=\dim(X)^{O(1)}$  is not known (see Section~\ref{sec:comments} below for more on this). Our proof of Theorem~\ref{thm:main intro} circumvents this problem about the linear structure of $\S_1$ by taking a different route. As we shall soon explain, this proof actually yields a stronger geometric conclusion (which is new even for dimension reduction in $\ell_1$) that does not follow from the approaches that were used in the literature~\cite{BC03,LN04,ACNNH11,Reg13}  to treat the $\ell_1$ setting.

$\S_1$ is of immense importance to mathematics, statistics and physics; it would be unrealistically ambitious to attempt to describe this here, but we shall now briefly indicate some of the multifaceted uses of $\S_1$ in combinatorics and computer science. The nuclear norm of the adjacency matrix of a  graph is also called~\cite{Gut78} its {\em graph energy}; see e.g.~the article~\cite{Nik07}, the monograph~\cite{LSG12} and the references therein for many applications which naturally give rise to a variety of algorithmic issues involving nuclear norm computations. The nuclear norm arises in many optimization scenarios, ranging from notable work~\cite{CR09,CT10,Rec11} on matrix completion and other  non-convex  optimization problems in matrix analysis and numerical linear algebra (e.g.~\cite{RFP10,DTV11}), differential privacy (e.g.~\cite{HLM12,LM13}), machine learning (e.g.~\cite{HDPDM12}), signal processing (e.g.~\cite{BSZ16}), computer vision (e.g.~\cite{GZZF14,GXMZFZ17}), sketching and data streams (e.g.~\cite{LW16,LW17}), and quantum computing (e.g.~\cite{Win04,HMS11}). In terms of direct relevance to dimension reduction, a natural question would be that of approximate nearest neighbor search in $\S_1$. This was posed explicitly in~\cite{And10} but resisted attempts to devise nontrivial data structures until the forthcoming work~\cite{ANNRW17}. The above cited work~\cite{HMS11} on quantum computing is  a direct precursor to the present article. Specifically, in~\cite{HMS11} the notion of quantum dimension reduction was introduced with the additional requirement that the nuclear norm-preserving embedding into low-dimensions is a quantum channel, and a strong impossibility result was obtained under this assumption (the additional structural information on the embedding makes the older approach of~\cite{CS02} applicable). Theorem~\ref{thm:main intro} (and even more so Theorem~\ref{thm:qs intro} below) complements this investigation by ruling out any sufficiently faithful low-dimensional embedding without any further restriction on its structure.

\subsection{Quotients of subsets} In what follows, the closed ball of radius $r\in [0,\infty)$ centered at a point $x$ of a metric space $(\MM,d_\MM)$ will be denoted $B_\MM(x,r)=\{y\in \MM:\ d_\MM(x,y)\le r\}$. Fix $\alpha\in [1,\infty)$. Following~\cite{Why58,Jam90,Gro99,BJLPS99}, a metric space $(\MM,d_\MM)$ is said to be an $\alpha$-Lipschitz quotient of a metric space $(\NN,d_\NN)$ if there is an {\em onto} mapping $\phi:\NN\twoheadrightarrow  \MM$ and (a scaling factor) $\lambda\in (0,\infty)$ such that
\begin{equation}\label{eq:def lip quo}
\forall\, x\in \NN,\ \forall r\in (0,\infty),\qquad B_\MM(\phi(x),\lambda r)\subset \phi\big(B_\NN(x,r)\big)\subset B_\MM(\phi(x),\alpha\lambda r).
\end{equation}
The second inclusion in~\eqref{eq:def lip quo} is just a rephrasing of the requirement that $\phi$ is Lipschitz, and the first inclusion in~\eqref{eq:def lip quo} means that $\phi$ is ``Lipschitzly open.'' For Banach spaces (and linear mappings), this definition is the dual of the bi-Lipschitz embedding requirement, i.e., given two Banach space $(X,\|\cdot\|_X), (Y,\|\cdot\|_Y)$, a linear mapping $T:X\to Y$ has distortion $\alpha$, i.e., $\lambda \|x\|_X\le \|Tx\|_Y\le \alpha\lambda \|x\|_X$ for all $x\in X$ and some $\lambda>0$, if and only if its adjoint $T^*:Y^*\to X^*$ is an $\alpha$-Lipschitz quotient. For general metric spaces, in lieu of duality one directly defines Lipschitz quotients as above.

In accordance with the Ribe program~\cite{Nao12,Bal13}, following insights from Banach space theory a natural way to weaken the notion of bi-Lipschitz embedding into a metric space $(\NN,d_\NN)$ is to study those metric spaces that are a Lipschitz quotient of a subset of $\NN$. Quantitatively, given a metric space $(\MM,d_\MM)$, denote by $\qs_\NN(\MM)$ the infimum over those $\alpha\in [1,\infty]$ for which there exists a subset $\mathcal{S}\subset \NN$ such that $(\MM,d_\MM)$ is an $\alpha$-Lipschitz quotient of $(\mathcal{S},d_\NN)$. The geometric meaning of this concept is elucidated via the following reformulation.  Given two nonempty subsets $U,V\subset \NN$, denote their minimal distance and Hausdorff distance, respectively, as follows.
$$
d_\NN(U,V)\eqdef \inf_{\substack{u\in U\\ v\in V}} d_\NN(u,v)\qquad\mathrm{and}\qquad \mathscr{H}_\NN(U,V)\eqdef \max\left\{\sup_{u\in U}\inf_{v\in V} d_\NN(u,v),\sup_{v\in V}\inf_{u\in U} d_\NN(u,v)\right\}.
$$
The following fact is straightforward to check directly from the definitions; see~\cite[Lemma~6.1]{MN04}.

\begin{fact}\label{fact:hausdorf}
Let $(\MM,d_\MM)$ and $(\NN,d_\NN)$ be metric spaces. The quantity $\qs_\NN(\MM)$ is equal to the infimum over those $\alpha\in [1,\infty]$ that satisfy the following property. One can assign to every $x\in \MM$ a nonempty subset $\sub_x$ of $\NN$ such that there exists (a scaling factor) $\lambda\in (0,\infty)$ for which
\begin{equation}\label{eq:hausdorff formulation}
\forall\, x,y\in \MM,\qquad \lambda \mathscr{H}_\NN(\sub_x,\sub_y)\le  d_\MM(x,y)\le \alpha \lambda d_\NN(\sub_x,\sub_y).
\end{equation}
\end{fact}
Clearly $\mathsf{qs}_\NN(\mathcal{M})\le \mathsf{c}_\NN(\mathcal{M})$ because if an embedding $f:\MM\to \NN$ satisfies~\eqref{eq:def distortion}, then by considering the singleton $\sub_x=\{f(x)\}\subset \NN$ for every $x\in \MM$ one obtains a collection of  subsets that satisfies~\eqref{eq:hausdorff formulation}. Hence, the following impossibility result for dimension reduction is stronger than Theorem~\ref{thm:main intro}.

\begin{theorem}\label{thm:qs intro}
There is a universal constant $c>0$ with the following property. For every $n\in \N$ there is an $n$-point subset $\sub \subset \ell_1\subset\S_1$ such that for every $\alpha\ge 1$ and every linear subspace $X$ of $\S_1$,
$$
\qs_X(\sub)\le \alpha \implies \dim(X)\ge n^{c/\alpha^2}.
$$
\end{theorem}

\subsection{Markov convexity}\label{sec:markov} The subtlety of proving results such as Theorem~\ref{thm:qs intro}, i.e., those  that provide limitations on the structure of subsets of quotients, is that one needs to somehow argue that no representation of $\MM$ using arbitrary subsets of $\NN$ can satisfy~\eqref{eq:hausdorff formulation}.  Note that $\mathsf{qs}_\NN(\mathcal{M})$ can be much smaller than $\mathsf{c}_\NN(\mathcal{M})$, and qualitatively the class of metric spaces that are Lipschitz quotients of subsets of $(\NN,d_\NN)$ is typically much richer than the class of metric spaces that admit a bi-Lipschitz embedding into $(\NN,d_\NN)$; a striking example of this is Milman's Quotient of Subspace Theorem~\cite{Mil85} that yields markedly stronger guarantees than the classical Dvoretzky theorem~\cite{Dvo60,Mil71} (for the purpose of this comparison, it suffice to consider the earlier work~\cite{Dil84}  that is weaker by a logarithmic factor, or even the bounds on the quotient of subspace problem in~\cite{FLM77}; one could also consider here the nonlinear results on quotients of subsets in~\cite{MN04} in comparison to their ``subset'' counterparts~\cite{BLMN05}).

Our proof of Theorem~\ref{thm:qs intro} (hence also Theorem~\ref{thm:main intro} as a special case) uses the bi-Lipschitz invariant {\em Markov convexity} that was introduced in~\cite{LNP06} and was shown in~\cite{MN08} to be preserved under Lipschitz quotients. Let $\{\chi_t\}_{t\in \Z}$ be a Markov chain on a state space $\Omega$. Given an integer $k\ge 0$, denote by $\{\widetilde{\chi}_t(k)\}_{t\in \Z}$ the process that equals $\chi_t$ for time $t\le k$, and evolves independently of $\chi_t$ (with respect to the same transition probabilities) for time $t>k$. Following~\cite{LNP06}, the Markov 2-convexity constant of a metric space $(\MM,d_\MM)$, denoted $\Pi_2(\MM)$, is the infimum over those $\Pi\in [0,\infty]$ such that for every Markov chain $\{\chi_t\}_{t\in \Z}$ on a state space $\Omega$ and every $f:\Omega\to \MM$ we have
\begin{equation}\label{eq:def markov convexity}
\sum_{k=1}^\infty\sum_{t\in \Z}\frac{1}{2^{2k}}\E\left[d_\MM \big(f\big(\widetilde{\chi}_t(t-2^k)\big),f(\chi_t)\big)^2\right]\le \Pi^2\sum_{t\in \Z} \E\left[d_\MM\big(f(\chi_t),f(\chi_{t-1})\big)^2\right].
\end{equation}

Because~\eqref{eq:def markov convexity} involves only pairwise distances, $\Pi_2(\mathcal{S})\le \Pi_2(\NN)$ for every metric space $(\NN,d_\NN)$ and any $\mathcal{S}\subset \NN$. Also, by~\cite{MN08} if for some $\alpha\in [1,\infty)$ a metric space $(\MM,d_\MM)$ is an $\alpha$-Lipschitz quotient of $(\NN,d_\NN)$, then $\Pi_2(\MM)\le \alpha\Pi_2(\NN)$.  As in~~\cite{MN08}, by combining these facts we see that Markov convexity yields the following obstruction to the existence of Lipschitz quotients from an arbitrary subset of $\NN$ onto $\MM$, or equivalently the existence of a representation of $\MM$ using subsets of $\NN$ as in~\eqref{eq:hausdorff formulation}.
\begin{equation}\label{quote:MN}
\qs_\NN(\MM)\ge \frac{\Pi_2(\MM)}{\Pi_2(\NN)}.
\end{equation}

The following theorem is the key structural contribution of the present article.
\begin{theorem}\label{thm:markov convexity}
Every finite-dimensional linear subspace $X$ of $\S_1$ satisfies $\Pi_2(X)\lesssim \sqrt{\log \dim(X)}$.
\end{theorem}
We deduce Theorem~\ref{thm:qs intro} from Theorem~\ref{thm:markov convexity} using another result of~\cite{MN08} which shows that there exists a sequence of connected series-parallel graphs $\{\mathsf{L}_k=(V(\mathsf{L}_k),E(\mathsf{L}_k)\}_{k=1}^\infty$ such that $\log |V(\mathsf{L}_k)|\asymp k$ and $\Pi_2(\mathsf{L}_k)\gtrsim \sqrt{k}$, where $\mathsf{L}_k$ is equipped with its shortest-path metric. The graphs $\{\mathsf{L}_k\}_{k=1}^\infty$ are known as the Laakso graphs~\cite{Laa00,LP01}, and the corresponding metric spaces  are even $O(1)$-doubling (see~\cite{Hei01} for the notion of doubling metric spaces; we do not need to use it here). These are not the same graphs as the ones that were used in~\cite{BC03} (though in~\cite{LMN04} the Laakso graphs were used as another way to rule out metric dimension reduction in $\ell_1$). The graphs of~\cite{BC03} are the diamond graphs $\{\mathsf{D}_k\}_{k=1}^\infty$ (see Figure~\ref{fig:diamond-laakso} for a depiction of $\mathsf{L}_3$ and $\mathsf{D}_3$), which are also series-parallel and one can show that the argument of~\cite{MN08} applies mutatis mutandis to yield the same properties for $\{\mathsf{D}_k\}_{k=1}^\infty$  as those that we stated   above for $\{\mathsf{L}_k\}_{k=1}^\infty$  (this is carried out in the forthcoming work~\cite{EMN17}). In any case, in~\cite{GNRS99} it was shown that any connected  series-parallel graph (equipped with its shortest path metric) embeds into $\ell_1$ (hence also into $\S_1$) with distortion $O(1)$. Let $\sub_k\subset \ell_1\subset \S_1$ be the image of such an embedding of $\mathsf{L}_k$. If $X$ is a finite-dimensional linear subspace of $\S_1$, then by combining  Theorem~\ref{thm:markov convexity} with~\eqref{quote:MN} and the fact~\cite{MN08} that $\Pi_2(\sub_k)\asymp \Pi_2(\mathsf{L}_k)\gtrsim \sqrt{\log |\sub_k|}$, we see that
\begin{equation}\label{eq:qs lower ratio}
\qs_X(\sub_k)\ge \frac{\Pi_2(\sub_k)}{\Pi_2(X)}\gtrsim \frac{\sqrt{\log|\sub_k|}}{\sqrt{\log \dim(X)}}.
\end{equation}
This simplifies to give Theorem~\ref{thm:qs intro}. As we explained above, the same conclusion holds for the images in $\ell_1$ that arise from an application of~\cite{GNRS99} to diamond graphs $\{\mathsf{D}_k\}_{k=1}^\infty$

\subsection{Comments on the proof of Theorem~\ref{thm:markov convexity}}\label{sec:comments} Having explained the ingredients of the proof of Theorem~\ref{thm:main intro}, we shall end this introduction by commenting on the proof of Theorem~\ref{thm:markov convexity}, stating additional consequences of this proof, and discussing limitations of previous methods in this context.

For  $m\in \N$ denote the subspace of $\S_1$ that consists of the $m\times m$ matrices by $\S_1^m$ (to be extra formal, one can think of these matrices as the top left $m\times m$ corner of infinite matrices corresponding to operators on $\ell_2$, with all the entries that are not in that corner vanishing). $\S_1^m$ is a subspace of $\S_1$ of dimension $m^2$, so Theorem~\ref{thm:markov convexity} implies in this special case that $\Pi(\S_1^m)\lesssim \sqrt{\log m}$. However, it is much simpler to prove this for $\S_1^m$ than to prove the full statement of Theorem~\ref{thm:markov convexity} for a general subspace of $\S_1$. Indeed, this property of $\S_1^m$ follows from a combination of results of~\cite{Pis75,BCL94,LNP06}. The conclusion of Theorem~\ref{thm:main intro} is therefore easier in the special case $X=\S_1^m$. As we explained above, since by~\cite{Sch87,BLM89,Tal90}  every finite-dimensional subspace $X$ of $\ell_1$ embeds with distortion $O(1)$ into $\ell_1^m$ for some $m=\dim(X)^{O(1)}$, for $\ell_1$ it suffices to prove the impossibility of metric dimension reduction when the target is the special subspace $\ell_1^m$ (as done in~\cite{BC03}) rather than a general subspace. However, it remains open whether or not every finite-dimensional subspace $X$ of $\S_1$ embeds with distortion $O(1)$ into $\S_1^m$ for some $m=\dim(X)^{O(1)}$. It isn't clear if it is reasonable to expect that such a phenomenon  holds in $\S_1$, because the proofs in~\cite{Sch87,BLM89,Tal90} rely on (substantial)   coordinate sampling arguments that seem to be inherently commutative and without a matricial interpretation.

We prove Theorem~\ref{thm:markov convexity} by showing directly that for every $q\in (1,2)$, any finite-dimensional subspace $X$ of $\S_1$ embeds into $\S_q$  with distortion $\dim(X)^{1-1/q}$; see Theorem~\ref{thm:embed into bigger q} below. This distortion is $O(1)$ when $q=1+1/\log\dim(X)$, so Theorem~\ref{thm:markov convexity} follows from the fact that $\Pi_2(\S_q)\lesssim 1/\sqrt{q-1}$ for every $q\in (1,2]$, which can be shown to hold true by combining results of~\cite{Pis75,BCL94,LNP06}.  The above estimate $\cc_X(\S_q)\le \dim(X)^{1-1/q}$ builds on a structural result of~\cite{Tom79} that is akin to an important lemma of Lewis~\cite{Lew78} in the commutative setting (namely for an $L_1(\mu)$ space instead of $\S_1$), in combination with matricial estimates that constitute the bulk of the technical part of our contribution. As an aside,  we provide a substantially simpler proof of a slight variant (that suffices for our purposes) of the aforementioned  noncommutative Lewis-like lemma of~\cite{Tom79} via a quick variational argument.

\begin{remark} The fact (Theorem~\ref{thm:embed into bigger q}) that any finite-dimensional linear subspace $X$ of $\S_1$ embeds with distortion $O(1)$ into $\S_{q}$ for $q=1+1/\log\dim(X)$ yields additional useful information beyond the estimate on $\Pi_2(X)$ of Theorem~\ref{thm:markov convexity}. Indeed, it directly implies that the {\em martingale cotype $2$ constant}  of $X$ is at most $\sqrt{\log \dim(X)}$; the definition of this invariant, which is due to~\cite{Pis86}, is recalled in Section~\ref{sec:main} below. By~\cite{Bal92,MN13} this implies that the {\em metric Markov cotype $2$ constant} (see~\cite{Bal92,MN13} for the relevant definition) of $X$ is $O(\sqrt{\log \dim(X)})$, which in turn  implies improved extension results for $X$-valued Lipschitz functions using~\cite{Bal92} as well as improved estimates for $X$-valued nonlinear spectral calculus using~\cite{MN10}. This also yields several improved Littlewood--Paley estimates~\cite{Xu98,MTX06,HN16} for $X$-valued functions and improved quantitative differentiation estimates for such functions~\cite{HN16}. Finally, using~\cite{LN14}, it yields improved vertical-versus-horizontal Poincar\'e inequalities for functions on the Heisenberg group that take values in low-dimensional subspaces of $\S_1$ (as an aside, it is natural to recall here the very interesting open question whether the Heisenberg group admits a bi-Lipschitz embedding into $\S_1$; see~\cite{NY17}). We shall not  include detailed statements of these  applications here because this would result in an exceedingly long digression, but one should note their availability.
\end{remark}

The impossibility result  for dimension reduction in $\ell_1$ was proved in~\cite{BC03} using linear programming duality. Different proofs were subsequently found in~\cite{LN04,ACNNH11,Reg13}. Specifically, the proof in~\cite{LN04} was a geometric argument (see also~\cite{LMN04} for a variant of the same idea for the Laakso graph), the proof in~\cite{ACNNH11} was a combinatorial argument (though inspired by the linear programming approach of~\cite{BC03}), and the proof in~\cite{Reg13} was an information-theoretical argument. Of these proofs, those of~\cite{BC03,ACNNH11,Reg13} rely on the coordinate structure of $\ell_1$ and do not seem to extend to the noncommutative setting of $\S_1$. The geometric approach of~\cite{LN04} (and its variants in~\cite{LMN04,JS09}) is more robust and could be used to deduce the impossibility of dimension reduction into $\S_1^m$ for small $m$, but not to obtain the full strength of Theorem~\ref{thm:main intro}. Also, we shall now explain why the method of~\cite{LN04} is inherently unsuited for obtaining the impossibility statement of Theorem~\ref{thm:qs intro} for quotients of subsets. To this end, we need to describe the bi-Lipschitz invariant that was used in~\cite{LN04} and is dubbed ``diamond convexity'' in~\cite{EMN17}. The iterative construction of the diamond graphs $\{\mathsf{D}_k\}_{k=1}^\infty$ (see Figure~\ref{fig:diamond-laakso}) replaces each edge $\{u,v\}$ in the $(k-1)$'th stage $\mathsf{D}_{k-1}$  by a quadrilateral $\{u,a,v,b\}$, i.e., $\{u,a\},\{a,v\}, \{v,b\}, \{b,u\}$ are the corresponding new edges in $\mathsf{D}_{k}$. The pair $\{a,b\}$ is  called a level-$k$ anti-edge and the set of all level-$k$ anti-edges  is denoted $\mathsf{A}_k$.

\begin{figure}[h]
\centering
\fbox{
\begin{minipage}{6.25in}
\centering
\smallskip
\includegraphics[scale=1.3]{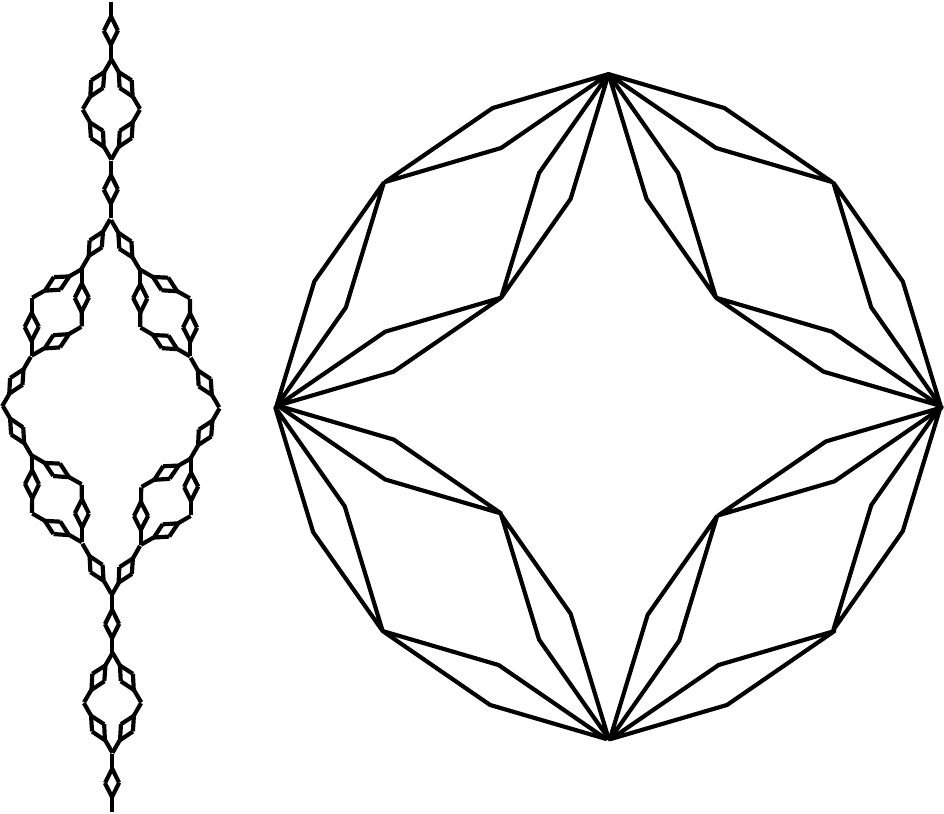}
\smallskip
\caption{\small \em The diamond graph $\mathsf{D}_3$ (on the right) and the Laakso graph $\mathsf{L}_3$ (on the left). Both the diamond graphs $\{\mathsf{D}_k\}_{k=1}^\infty$ and the Laakso graphs $\{\mathsf{L}_k\}_{k=1}^\infty$ are defined iteratively as follows, starting with $\mathsf{D}_1=\mathsf{L}_1$ being a single edge. To pass from $\mathsf{D}_k$ to $\mathsf{D}_{k+1}$, replace each edge of $\mathsf{D}_k$ by two parallel paths of length $2$. To pass from $\mathsf{L}_k$ to $\mathsf{L}_{k+1}$, subdivide each edge of $\mathsf{L}_k$ into a path of length $4$, remove the middle two edges in this path, and replace them by two parallel  paths of length $2$.
}
\label{fig:diamond-laakso}
\end{minipage}
}
\end{figure}

Following~\cite{EMN17}, the {\em diamond $2$-convexity constant} of a metric space $(\MM,d_\MM)$, denoted $\Delta_2(\MM)$, is the infimum over those $\Delta\in (0,\infty]$ such that for every $k\in \N$, every $f:V(\mathsf{D}_k)\to \MM$ satisfies
$$
\sum_{j=1}^k\sum_{\{a,b\}\in \mathsf{A}_j} d_\MM\big(f(a),f(b)\big)^2\le \Delta^2\sum_{\{u,v\}\in E(\mathsf{D}_k)}d_\MM\big(f(u),f(v)\big)^2.
$$
With this terminology, the proof of~\cite{LN04} derives an upper bound on $\Delta_2(\ell_1^m)$ and contrasts  it with $\Delta_2(\mathsf{D}_k)$, allowing one to deduce that $\cc_{\ell_1^m}(\mathsf{D}_k)$ must be large if $m$ is small, similarly to the way we used the Markov $2$-convexity constant in~\eqref{eq:qs lower ratio}. But, working with diamond convexity  cannot yield impossibility results for quotients of subsets as in Theorem~\ref{thm:qs intro}, because in~\cite{EMN17} it is shown that there exist metric spaces $(\MM,d_\MM)$ and $(\NN,d_\NN)$ such that $\MM$ is a Lipschitz quotient of $\NN$  yet $\Delta_2(\NN)<\infty$ and $\Delta_2(\MM)=\infty$. In other words, in contrast to Markov convexity, diamond $2$-convexity is not preserved under Lipschitz quotients. Thus, working with Markov $2$-convexity as we do here has an advantage over the approach of~\cite{LN04} by yielding Theorem~\ref{thm:qs intro} whose statement is new  even for $\ell_1$.



\medskip

\noindent{\bf Acknowledgements.} We thank A. Andoni, R. Krauthgamer and M. Mendel for helpful input.



\section{Proof of Theorem~\ref{thm:markov convexity}}\label{sec:main}

We shall start by recording for ease of later reference some basic notation and well-known facts about Schatten--von Neumann trace classes that will be used repeatedly in what follows. The  standard material that appears below can be found in many texts, including e.g.~\cite{Dix96,Bha97,Sim05,Car10}.

Throughout, the Hilbert space $\ell_2$ will be over the real\footnote{For many purposes it is important to work with complex scalars, and correspondingly complex matrices. However, for the purpose of the ensuing metric results, statements over $\R$ are equivalent to their complex counterparts.} scalar field $\R$. The standard scalar product on $\ell_2$ is $\langle\cdot,\cdot\rangle :\ell_2\times\ell_2\to \R$. Given a closed linear subspace $V$ of $\ell_2$, the orthogonal projection onto $V$ will be denoted $\proj_V:\ell_2\to \ell_2$ and the orthogonal complement of $V$ will be denoted $V^\perp$. The group of orthogonal operators on $\ell_2$ is denoted $\mathsf{O}_{\ell_2}$ and the set of compact operators $T:\ell_2\to \ell_2$ is denoted $\mathsf{K}_{\ell_2}$. The elements of $\mathsf{K}_{\ell_2}$ are characterized as those operators $T:\ell_2\to \ell_2$ that admit a singular value decomposition, i.e., they can be written as $T=U\Sigma V$, where $U,V\in \mathsf{O}_{\ell_2}$ and $\Sigma:\ell_2\to \ell_2$ is a diagonal operator (say, relative to the standard coordinate basis $\{e_i\}_{i=1}^\infty$ of $\ell_2$) with nonnegative entries that tend to $0$. The diagonal entries of $\Sigma$ are called the singular values of $T$ and their decreasing rearrangement is denoted $\sigma_1(T)\ge \sigma_2(T)\ge \cdots$. Note that $\sqrt{T^*T}=V^{-1}\Sigma V$ and $\sqrt{TT^*}=U\Sigma U^{-1}$, and so $T=UV\sqrt{T^*T}=\sqrt{TT^*}UV$ (polar decompositions).

Given $\beta\in (0,\infty)$ and a symmetric positive semidefinite operator $T\in \mathsf{K}_{\ell_2}$, the power $T^\beta\in \mathsf{K}_{\ell_2}$ is defined via the usual functional calculus, i.e., if $T=U\Sigma U^{-1}$ is the singular value decomposition of $T$, then $T^\beta=U\Sigma^\beta U^{-1}$, where $\Sigma^\beta$ is obtained from the  operator $\Sigma$ (which we recall is diagonal with nonnegative entries) by raising each of its  entries to the power $\beta$. In what follows, it will also be very convenient to adhere to (and make frequent use of) the following convention for negative powers of symmetric positive semidefinite operators . If the diagonal of $\Sigma$ is $(\sigma_1,\sigma_2,\ldots)\in [0,\infty)^{\aleph_0}$, then let $\Sigma^{-\beta}$ be the diagonal operator  whose $i$'th diagonal entry equals $0$ if $\sigma_i=0$ and equals $1/\sigma_i^\beta$ if $\sigma_i>0$. Then, write $T^{-\beta}=U\Sigma^{-\beta} U^{-1}$. Observe that if $T$ is invertible in addition to being symmetric and positive semidefinite, then under this convention  $T^{-1}$ coincides with the usual inverse of $T$. But, in general we have $TT^{-1}=T^{-1}T=\proj_{\ker(T)^\perp}$, where $\ker(T)$ is the kernel of $T$.

An operator $T\in \mathsf{K}_{\ell_2}$ is said to be nuclear if $\sum_{j=1}^\infty\sigma_j(T)<\infty$. In this case, the trace of $T$ is well defined as $\trace(T)=\sum_{j=1}^\infty \langle Te_j,e_j\rangle$.  Given $p\in (0,\infty)$, the Schatten--von Neumann trace class $\S_p$ is the space of all $T\in  \mathsf{K}_{\ell_2}$ whose singular values are $p$-summable, in which case one defines $\|T\|_{\S_p}$  by
\begin{equation}\label{eq:def schatten p}
\|T\|_{\S_p}\eqdef \bigg(\sum_{j=1}^\infty \sigma_j(T)^p\bigg)^{\frac{1}{p}}=\left(\trace\left[\left(T^*T\right)^{\frac{p}{2}}\right]\right)^{\frac{1}{p}}=\left(\trace\left[\left(TT^*\right)^{\frac{p}{2}}\right]\right)^{\frac{1}{p}}.
\end{equation}
When $p=\infty$ the quantity $\|T\|_{\S_\infty}=\sup_{j\in \N} \sigma_j(T)$ is the operator norm of $T$. If $p\in [1,\infty]$, then $\|\cdot\|_{\S_p}$ is a norm; the (non-immediate) proof of this fact is a classical theorem of von Neumann~\cite{vN37}.

The Schatten--von Neumann norms are invariant under the group $\mathsf{O}_{\ell_2}$, i.e., for all $p\in (0,\infty)$,
\begin{equation}\label{eq:unitarily invariant}
\forall\, U,V\in\mathsf{O}_{\ell_2},\ \forall\, T\in \S_p,\qquad \|UTV\|_{\S_p}=\|T\|_{\S_p}.
\end{equation}
The von Neumann trace inequality~\cite{vN37} (see also~\cite{Mir75}) asserts that every $S,T\in \mathsf{K}_{\ell_2}$ satisfy
\begin{equation}\label{eq:von neumann}
\trace(ST)\le \sum_{j=1}^\infty \sigma_j(S)\sigma_j(T).
\end{equation}
This implies in particular that if $S\in \S_1$ is positive semidefinite and $T\in \mathsf{K}_{\ell_2}$, then
\begin{equation}\label{eq:ideal property}
\trace[ST]=\trace[TS]\le \|T\|_{\S_\infty}\trace[S].
\end{equation}
Also, by trace duality (see e.g.~\cite[Theorem~7.1]{Car10}), the von Neumann inequality~\eqref{eq:von neumann} implies the H\"older inequality for Schatten--von Neumann norms (see e.g.~\cite[Corollary~IV.2.6]{Bha97}), which asserts that if $a,b,c\in [1,\infty]$ satisfy $\frac{1}{c}=\frac{1}{a}+\frac{1}{b}$, then for every $A\in \S_a$ and $B\in \S_b$ we have
\begin{equation}\label{eq:holder schatten}
\|AB\|_{\S_c}\le \|A\|_{\S_a}\|B\|_{\S_b}.
\end{equation}

For  $m\in \N$, the above discussion can be repeated mutatis mutandis with the infinite dimensional Hilbert space $\ell_2$ replaced by the $m$-dimensional Euclidean space $\ell_2^m$. In this setting, we denote the corresponding Schatten--von Neumann matrix spaces by $\S_p^m$ for every $p\in (0,\infty)$. We shall also use the standard notations $\mathsf{K}_{\ell_2^m}=\mathsf{M}_m(\R)$ and $\mathsf{O}_{\ell_2^m}=\mathsf{O}_m$. Some of the ensuing arguments  are carried out for linear subspaces of $\S_p^m$ rather than for arbitrary finite-dimensional linear subspaces of $\S_p$.  We suspect that this restriction is only a matter of convenience and it could be removed (specifically, in Lemma~\ref{lem:lewis} below), but for our  purposes  it suffices to treat subspaces of $\S_p^m$ by the following simple and standard lemma (a truncation argument). Since we could not locate a clean reference for this statement, we include its straightforward proof in Remark~\ref{rem:schauder} below.

\begin{lemma}\label{lem:discretization}
Fix $p\in [1,\infty)$ and suppose that $X$ is a finite-dimensional linear subspace of $\S_p$. For every $\e\in (0,1)$ there exist an integer  $m=m(X,\e)\in \N$ and a linear operator $\mathsf{J}:X\to \S_p^m$ such that
\begin{equation}\label{eq:J isomorph}
\forall\, A\in X,\qquad (1-\e)\|A\|_{\S_p}\le\|\mathsf{J}A\|_{\S_p^m}\le \|A\|_{\S_p}.
\end{equation}
\end{lemma}

For ease of later reference, we shall record the following general lemma. In it, as well as in the subsequent discussion, the usual PSD partial order is denoted by $\le$, i.e., given two symmetric bounded operators $S,T:\ell_2\to \ell_2$ the notation $S\le T$ means that $T-S$ is positive semidefinite.

\begin{lemma}\label{lem:1/2} Let $S,T\in \mathsf{K}_{\ell_2}$ be symmetric positive semidefinite operators such that $S\le T$. Then
$$
 0<\beta\le \frac12\implies \left\|S^\beta T^{-\beta}\right\|_{\S_\infty}\le 1.
$$
\end{lemma}

\begin{proof}
Fix  $\beta\in (0,1/2]$. Then $2\beta\in (0,1]$ and therefore the classical L\"owner theorem~\cite{Low34} (see e.g.~\cite[Theorem~V.1.9]{Bha97}) asserts that the function $u\mapsto u^{ 2\beta}$ is operator monotone on $[0,\infty)$. Hence the assumption $S\le T$ implies that $S^{2\beta}\le T^{2\beta}$, i.e.,
\begin{equation}\label{eq:2 beta}
\forall\, x\in \ell_2,\qquad \langle S^{2\beta}x,x\rangle\le \langle T^{2\beta}x,x\rangle.
\end{equation}
While $T$ need not be invertible, we have $T^{-\beta}T^{2\beta}T^{-\beta}=\proj_{\ker(T)^\perp}$. Hence, for every $y\in \ell_2$ we have
\begin{multline*}
\|y\|_{\ell_2^2}^2\ge \left\|\proj_{\ker(T)^\perp}y\right\|_{\ell_2}^2=\left\langle \proj_{\ker(T)^\perp}y,y\right\rangle =\left\langle T^{-\beta}T^{2\beta}T^{-\beta}y,y\right\rangle =
\left\langle T^{2\beta}T^{-\beta}y,T^{-\beta}y\right\rangle\\
\stackrel{\eqref{eq:2 beta} }{\ge} \left\langle S^{2\beta}T^{-\beta}y,T^{-\beta}y\right\rangle=\left\langle S^{\beta}T^{-\beta}y,S^\beta T^{-\beta}y\right\rangle=\left\|S^{\beta}T^{-\beta}y\right\|_{\ell_2}^2.
\end{multline*}
Thus $\left\|S^{\beta}T^{-\beta}y\right\|_{\ell_2}\le \|y\|_{\ell_2}$ for all $y\in \ell_2$, which is the desired conclusion.
\end{proof}

Our proof of Theorem~\ref{thm:markov convexity} relies on Lemma~\ref{lem:lewis} below, which is a useful structural result for subspaces of $\S_p^m$.  In fact, we will only need the case $p=1$ of  Lemma~\ref{lem:lewis}, but we include its proof for general $p\in (0,\infty)$ because this does not require additional effort beyond the special case $p=1$.

Lemma~\ref{lem:lewis} is a noncommutative analogue of an important classical lemma that was proved by Lewis in~\cite{Lew78} for (finite-dimensional linear subspaces of) $L_p(\mu)$ spaces. The Lewis lemma was extended by Tomczak-Jaegermann in~\cite{Tom79} in a different manner to both  Banach lattices and Schatten--von Neumann classes. In particular, Theorem~2.3 of~\cite{Tom79} states a slightly different noncommutative Lewis-type lemma for $\S_p$ when $1<p<\infty$ (note that this is proved in~\cite{Tom79} only when $2\le p<\infty$ since only that  range is needed in~\cite{Tom79}). The variant that is stated in~\cite{Tom79} would suffice for our purposes as well, but we include a different proof here because the argument of~\cite{Tom79} is significantly more sophisticated than the way we proceed below. As an aside, our proof applies also to the range $0<p<1$ while the proof in~\cite{Tom79} does not because it relies inherently on duality. The need to obtain such a result for $L_p(\mu)$ spaces when $0<p<1$ arose in~\cite{SZ01}, where a new  proof of the Lewis lemma was obtained so as to be applicable to these values of $p$ (Lewis' argument in~\cite{Lew78} also relied on duality, hence requiring $p\ge 1$). This generalization turned out to lead to a simpler approach, and our proof below consists of a noncommutative adaptation of the argument of~\cite{SZ01}.

\begin{lemma}[Lewis-type basis for subspaces of $\S_p^m$]\label{lem:lewis} Fix $p\in (0,\infty)$ and $k,m\in \N$. Let $X$ be a  linear subspace of $\S_p^m$ with $\dim(X)=k$. Then there exists a basis $\{T_1,\ldots,T_k\}$ of $X$ such that if we define
\begin{equation}\label{eq:def M}
M\eqdef \sum_{i=1}^k T_i^*T_i\in \S_p^m,
\end{equation}
then for all $i,j\in \{1,\ldots,k\}$, denoting by $\d_{ij}$ the Kronecker delta, we have
\begin{equation}\label{eq:delta ij}
\trace\left[\frac12\left(T_i^*T_j+T_j^*T_i\right)M^{\frac{p}{2}-1}\right]=\d_{ij}.
\end{equation}
\end{lemma}

\begin{proof} Fix an arbitrary basis $\{W_1,\ldots,W_k\}$ of $X$. For every matrix $A=(a_{st})\in \mathsf{M}_k(\R)$ define
\begin{equation}\label{eq:def SiB}
\forall\, j\in \{1,\ldots,k\},\qquad S_j(A)\eqdef \sum_{u=1}^k a_{ju} W_u\in X,
\end{equation}
and
\begin{equation}\label{eq:def L}
\Lambda(A)\eqdef \bigg(\sum_{j=1}^k S_j^*(A)S_j(A)\bigg)^{\frac12}\in \S_p^m.
\end{equation}
Since $W_1,\ldots,W_k$ are linearly independent, $\Lambda(A)=0\iff A=0$. It follows that $A\mapsto \|\Lambda(A)\|_{\S_p^m}$ is equivalent to a  norm on $\mathsf{M}_k(\R)$. Indeed,  $\|\Lambda(A)\|_{\S_p^m}\asymp_{p,m} \|\Lambda(A)\|_{\S_2^m}$ and one computes directly that  $A\mapsto \|\Lambda(A)\|_{\S_2^m}$ is a Hilbertian semi-norm on $\mathsf{M}_k(\R)$. Hence, if we define
\begin{equation}\label{eq:def psi}
\psi(A)\eqdef \|\Lambda(A)\|_{\S_p^m}^p=\trace\left[\Lambda(A)^p\right],
\end{equation}
then the set $\{A\in \mathsf{M}_k(\R):\ \psi(A)=1\}$ is compact.  Therefore the continuous mapping $A\mapsto\det(A)$ attains its maximum on this set, so  fix from now on some $B=(b_{st})\in \mathsf{M}_k(\R)$ such that
$$
\det(B)=\max_{\stackrel{A\in\mathsf{M}_k(\R)}{ \psi(A)=1}}\det(A).
$$

Because $\det(B)>0$, we will soon explain that $\psi$ is continuously differentiable on a neighborhood of $B$. Therefore there exists $\lambda\in \R$ (a Lagrange multiplier) such that $(\nabla \det)(B)=\lambda(\nabla\psi)(B)$. A standard formula for the gradient of the determinant (which follows directly from the cofactor expansion) asserts that $(\nabla \det)(B)=\det(B) (B^*)^{-1}$. We will also soon compute that
\begin{equation}\label{eq:grad identity}
\forall\, u,t\in \{1,\ldots,k\},\qquad \big((\nabla \psi)(B)\big)_{ut}= \frac{p}{2}\trace\left[\big(S_u(B)^*W_t+W_t^*S_{u}(B)\big)\Lambda(B)^{p-2}\right].
\end{equation}
Therefore, the above Lagrange multiplier identity asserts that for all $u,t\in \{1,\ldots,k\}$,
\begin{equation}\label{eq:coordinate identity}
\big((B^*)^{-1}\big)_{ut}=\frac{\lambda p}{2\det(B)}\trace\left[\big(S_u(B)^*W_t+W_t^*S_{u}(B)\big)\Lambda(B)^{p-2}\right].
\end{equation}
Fix $v\in \{1,\ldots,k\}$, multiply~\eqref{eq:coordinate identity} by $(B^*)_{tv}=b_{vt}$ and sum over $t\in \{1,\ldots,k\}$, thus arriving at
\begin{equation}\label{eq:Si version}
\d_{uv}=\frac{\lambda p}{2\det(B)}\trace\left[\big(S_u(B)^*S_v(B)+S_v(B)^*S_{u}(B)\big)\Lambda(B)^{p-2}\right].
\end{equation}
By summing~\eqref{eq:Si version} over $u=v\in \{1,\ldots,k\}$ while recalling the definition of the matrix $\Lambda(B)$ in~\eqref{eq:def L}, we see that $0<2n\det(B)=p\lambda\trace[\Lambda(B)^{p/2}]$. Since $\Lambda(B)$ is positive semidefinite, it follows that $\lambda>0$.  Hence, the assertions of Lemma~\ref{lem:lewis} hold true for $\{T_j=(\lambda p/\det(B))^{1/p} S_j(B)\}_{j=1}^k$.

It remains to verify that $\psi$ is continuously differentiable and the identity~\eqref{eq:grad identity} holds true; this is a standard exercise in spectral calculus which we include for completeness.  Consider the subspace $V=\ker(W_1)\cap \ldots\cap \ker(W_1)\subset \ell_2^m$ and observe that for every invertible matrix $A\in \mathsf{GL}_k(\R)$  the definitions~\eqref{eq:def SiB} and~\eqref{eq:def L} of $S_1(A),\ldots,S_k(A)$ and $\Lambda(A)$, respectively, imply that we also have $V=\ker(S_1(A))\cap\ldots\cap \ker(S_k(A))=\ker(\Lambda(A))$. So, for all $A\in \mathsf{GL}_k(\R)$ the restrictions of $\Lambda(A)$ to $V$ and $V^\perp$ satisfy $\Lambda(A)|_V=0$ and $\Lambda(A)_{V^{\perp}}\in \mathsf{GL}(V^\perp)$, respectively (the latter assertion is that $\Lambda(A)_{V^{\perp}}:V^\perp\to V^{\perp}$ is invertible). Fix $\e\in (0,\infty)$ that is strictly smaller than the smallest nonzero eigenvalue of $\Lambda(B)^2$ and let $\mathcal{D}$ be a simply connected open domain that is contained in the complex half-plane $\{z\in\C:\ \Re(z)>\e\}$ and contains all of the nonzero eigenvalues of $\Lambda(B)^2$. By continuity of the mapping $\Lambda:\mathsf{M}_k(\R)\to \S_p^m$ and the fact that $B$ is invertible, the above reasoning implies that there exists an open neighborhood $\mathcal{O}\subset\mathsf{M}_k(\R)$ of $B$  such that every $A\in \mathcal{O}$ is invertible and all of the nonzero eigenvalues of $\Lambda(A)^2$ are contained in $\mathcal{D}$. Since the function $z\mapsto z^{p/2}$  is analytic on $\mathcal{D}$, by the Cauchy integral formula (for both this function and its derivative), every $A\in \mathcal{O}$ satisfies
\begin{equation}\label{eq:zp}
\Lambda(A)^p=\big(\Lambda(A)^2\big)^{\frac{p}{2}}=\frac{1}{2\pi \mathsf{i}}\oint_{\partial \mathcal{D}} z^{\frac{p}{2}}\big(z\mathsf{Id}_m-\Lambda(A)^2\big)^{-1}\ud z,
\end{equation}
and
\begin{equation}\label{eq z p-2}
\frac{p}{2}\Lambda(A)^{p-2}=\frac{p}{2}\big(\Lambda(A)^2\big)^{\frac{p}{2}-1}=\frac{1}{2\pi \mathsf{i}} \oint_{\partial \mathcal{D}} z^{\frac{p}{2}}\big(z\mathsf{Id}_m-\Lambda(A)^2\big)^{-2}\ud z,
\end{equation}
where $\mathsf{Id}_m\in \mathsf{M}_m(\R)$ is the identity matrix.  It is important to note that~\eqref{eq z p-2} respects our convention for negative powers of  symmetric positive semidefinite matrices that are not invertible, because the nonzero eigenvalues of any such $A$ are contained in $\mathcal{D}$, and also $0\in \C\setminus \overline{\mathcal{D}}$ so that the Cauchy integral vanishes on the kernel $V$. Recalling~\eqref{eq:def L}, the (quadratic) mapping $\Lambda^2:\mathsf{M}_k(\R)\to \S_p^m$ is continuously differentiable, and in fact for every  $u,t\in \{1,\ldots,k\}$ we can compute directly that
\begin{equation}\label{eq:differentiate L2}
\frac{\partial \Lambda^2}{\partial a_{ut}}(A)\stackrel{\eqref{eq:def SiB}\wedge\eqref{eq:def L}}{=}\frac{\partial }{\partial a_{ut}}\bigg(\sum_{j=1}^k \sum_{\alpha=1}^k\sum_{\beta=1}^k a_{j\alpha}a_{j\beta}W_\alpha^*W_\beta\bigg)(A)\stackrel{\eqref{eq:def SiB}}{=}S_u(A)^*W_t+W_t^*S_u(A).
\end{equation}
It therefore follows from the integral representation~\eqref{eq:zp}  that the function $\Lambda^p:\mathsf{M}_k(\R)\to \S_p^m$ is continuously differentiable on the neighborhood $\mathcal{O}$ of $B$, and moreover for every $u,t\in \{1,\ldots,k\}$,
\begin{multline}\label{eq:differentiate integral}
\frac{\partial \Lambda^p}{\partial a_{ut}}(B)\stackrel{\eqref{eq:zp} }{=}\frac{1}{2\pi \mathsf{i}}\oint_{\partial \mathcal{D}} z^{\frac{p}{2}}\big(z\mathsf{Id}_m-\Lambda(B)^2\big)^{-1} \frac{\partial\Lambda^2}{\partial a_{ut}}(B) \big(z\mathsf{Id}_m-\Lambda(B)^2\big)^{-1}\ud z\\\stackrel{\eqref{eq:differentiate L2}}{=}\frac{1}{2\pi \mathsf{i}}\oint_{\partial \mathcal{D}} z^{\frac{p}{2}}\big(z\mathsf{Id}_m-\Lambda(B)^2\big)^{-1} \big(S_u(B)^*W_t+W_t^*S_{u}(B)\big) \big(z\mathsf{Id}_m-\Lambda(B)^2\big)^{-1}\ud z,
\end{multline}
where we used the fact that for every invertible matrix function $(s\in \R)\mapsto  C(s)\in\mathsf{GL}_m(\R)$ we have $(C(s)^{-1})'=-C(s)^{-1}C'(s)C(s)^{-1}$ (as seen  by differentiating the identity $C(s)C(s)^{-1}=\Id_m$). Now,
\begin{eqnarray}
\nonumber (\nabla \psi)(B)_{ut}&\stackrel{\eqref{eq:def psi}}{=}&\trace\left[\frac{\partial \Lambda^p}{\partial a_{ut}}(B)\right]\\ \nonumber
 &\stackrel{\eqref{eq:differentiate integral}}{=}&\frac{1}{2\pi \mathsf{i}} \oint_{\partial \mathcal{D}} z^{\frac{p}{2}}\trace\left[\big(z\mathsf{Id}_m-\Lambda(B)^2\big)^{-1} \big(S_u(B)^*W_t+W_t^*S_{u}(B)\big) \big(z\mathsf{Id}_m-\Lambda(B)^2\big)^{-1}\right]\ud z\\ \label{eq:cyclicity}
 &=& \frac{1}{2\pi \mathsf{i}} \oint_{\partial \mathcal{D}} z^{\frac{p}{2}}\trace\left[\big(S_u(B)^*W_t+W_t^*S_{u}(B)\big) \big(z\mathsf{Id}_m-\Lambda(B)^2\big)^{-2} \right]\ud z\\
 &=& \nonumber \trace\left[\big(S_u(B)^*W_t+W_t^*S_{u}(B)\big)\bigg(\frac{1}{2\pi \mathsf{i}} \oint_{\partial \mathcal{D}} z^{\frac{p}{2}}\big(z\mathsf{Id}_m-\Lambda(B)^2\big)^{-2}\ud z\bigg)\right]\\
 &\stackrel{\eqref{eq z p-2}}{=}&  \frac{p}{2}\trace\left[\big(S_u(B)^*W_t+W_t^*S_{u}(B)\big)\Lambda(B)^{p-2}\right], \nonumber
\end{eqnarray}
where in~\eqref{eq:cyclicity} we used the cyclicity of the trace. This concludes the verification of~\eqref{eq:grad identity}.
\end{proof}

We shall now proceed to derive several additional lemmas as consequences of Lemma~\ref{lem:lewis}. These lemmas are steps towards the proof of Theorem~\ref{thm:embed into bigger q} below, which is the  main result of this section.

\begin{lemma}\label{eq:upper p} Fix $p,q\in (0,\infty)$ with $p<q$. Continuing with the notation of Lemma~\ref{lem:lewis}, we have
\begin{equation}
\label{eq:A upper lewis}
\forall\, A\in  X,\qquad \|A\|_{\S_p^m}\le k^{\frac{1}{p}-\frac{1}{q}}\left\|AM^{\frac{p-q}{2q}}\right\|_{\S_q^m}.
\end{equation}
\end{lemma}

\begin{proof} Observe that
\begin{equation}\label{eq:trace M identity}
\left\|M^{\frac{q-p}{2q}}\right\|_{\S_{\frac{pq}{q-p}}^m}^{\frac{pq}{q-p}}\stackrel{\eqref{eq:def schatten p}}{=}\trace \left[M^{\frac{p}{2}}\right]\stackrel{\eqref{eq:def M}}{=}\sum_{i=1}^k \trace \left[T_i^*T_i M^{\frac{p}{2}-1}\right]\stackrel{\eqref{eq:delta ij}}{=}k.
\end{equation}
The definition~\eqref{eq:def M} of $M$ implies that $\ker(M)=\cap_{i=1}^k\ker(T_i)$. Since $A\in X=\spn(\{T_1,\ldots,T_k\})$, it follows that $\ker(A)\supset \ker(M)$. Recalling our convention for negative powers of symmetric positive semidefinite matrices that need not be invertible, this implies that $A=AM^{(p-q)/(2q)}M^{(q-p)/(2q)}$.   Hence, since $1/p=1/q+(q-p)/(pq)$, we conclude the proof of Lemma~\ref{eq:upper p} as follows.
\begin{equation*}
\|A\|_{\S_p^m}=\left\|AM^{\frac{p-q}{2q}}M^{\frac{q-p}{2q}}\right\|_{\S_p^m}\stackrel{\eqref{eq:holder schatten}}{\le} \left\|AM^{\frac{p-q}{2q}}\right\|_{\S_q^m}\left\|M^{\frac{q-p}{2q}}\right\|_{\S_{\frac{pq}{q-p}}^m}\stackrel{\eqref{eq:trace M identity}}{=} k^{\frac{1}{p}-\frac{1}{q}}\left\|AM^{\frac{p-q}{2q}}\right\|_{\S_q^m}.\qedhere
\end{equation*}
\end{proof}

\begin{lemma}\label{lem:beta} Fix $p,\beta\in (0,\infty)$ with $\beta\le \frac12$.  Continuing with the notation of Lemma~\ref{lem:lewis}, we have
\begin{equation}\label{eq:desired beta in lemma}
\forall\, A\in X\subset \S_p^m,\qquad \left\|\left(A^*A\right)^{\beta}M^{-\beta}\right\|_{\S_\infty^m}\le  \left(\trace\left[A^*A M^{\frac{p}{2}-1}\right]\right)^{\beta}.
\end{equation}
\end{lemma}

\begin{proof} Fix $A\in X$. Since $\{T_1,\ldots,T_k\}$ is a basis of $X$, we can write $A=\sum_{j=1}^ka_j T_j$ for some scalars $a_1,\ldots,a_k\in \R$. Observe that for every $x\in \R^m$ we have
\begin{multline*}
\langle A^*A x,x\rangle =\bigg\|\sum_{i=1}^k a_i T_ix\bigg\|_{\ell_2^m}^2 \le \bigg(\sum_{i=1}^k |a_i|\cdot\|T_ix\|_{\ell_2^m}\bigg)^2\le \bigg(\sum_{i=1}^k a_i^2\bigg)\sum_{i=1}^k \|T_ix\|_{\ell_2^m}^2\stackrel{\eqref{eq:def M}}{=} \bigg(\sum_{i=1}^k a_i^2\bigg)\langle Mx,x\rangle.
\end{multline*}
Hence the following matrix inequality holds true in the PSD order.
\begin{equation}\label{eq:|A|2}
A^*A\le \bigg(\sum_{i=1}^k a_i^2\bigg)M.
\end{equation}
By Lemma~\ref{lem:1/2}, which is where we are using the assumption $0<\beta\le \frac12$, it follows from~\eqref{eq:|A|2} that
\begin{equation}\label{eq:use beta lemma}
\left\|\left(A^*A\right)^\beta M^{-\beta}\right\|_{\S_\infty^m}\le \bigg(\sum_{i=1}^k a_i^2\bigg)^\beta.
\end{equation}
The desired estimate~\eqref{eq:desired beta in lemma} is equivalent to~\eqref{eq:use beta lemma} due to the following identity.
\begin{multline*}
\trace\left[A^*AM^{\frac{p}{2}-1}\right]=\trace\left[\bigg(\sum_{i=1}^k a_iT_i^*\bigg)\bigg(\sum_{j=1}^k a_jT_j\bigg)M^{\frac{p}{2}-1}\right]\\=\sum_{i=1}^n\sum_{j=1}^n a_ia_j\trace\left[\frac12\left(T_i^*T_j+T_j^*T_i\right)M^{\frac{p}{2}-1}\right]\stackrel{\eqref{eq:delta ij}}{=}\sum_{i=1}^k a_i^2.\tag*{\qedhere}
\end{multline*}
\end{proof}

Thus far our reasoning worked for all $p\in (0,\infty)$, but the following lemma requires that $p\ge 1$.

\begin{lemma}\label{lem:upper lewis} Fix $p,q\in [1,\infty)$ with $p<q$. Continuing with the notation of Lemma~\ref{lem:lewis}, we have
\begin{equation}\label{eq:span upper AM}
\forall\, A\in X,\qquad  \left\|AM^{\frac{p-q}{2q}}\right\|_{\S_q^m}\le \max\left\{k^{\frac{p-2}{2}\left(\frac{1}{p}-\frac{1}{q}\right)},1\right\} \|A\|_{\S_p^m}.
\end{equation}
\end{lemma}

\begin{proof} Fix any $A\in \S_p^m$. Writing $A=U\sqrt{A^*A}$ for some orthogonal matrix $U\in \mathsf{O}_m$, we have
\begin{multline}\label{eq:unitary}
\left\|AM^{\frac{p-q}{2q}}\right\|_{\S_q^m}=\left\|U\left(A^*A\right)^{\frac12}M^{\frac{p-q}{2q}}\right\|_{\S_q^m}
\stackrel{\eqref{eq:unitarily invariant}}{=}\left\|\left(A^*A\right)^{\frac{p}{2q}}\left(A^*A\right)^{\frac{q-p}{2q}}M^{\frac{p-q}{2q}}\right\|_{\S_q^m}\\ \stackrel{\eqref{eq:ideal property}}{\le}
\left\|\left(A^*A\right)^{\frac{p}{2q}}\right\|_{\S_q^m}\left\|\left(A^*A\right)^{\frac{q-p}{2q}}M^{\frac{p-q}{2q}}\right\|_{\S_\infty^m}
\stackrel{\eqref{eq:def schatten p}}{=}\|A\|_{\S_p^m}^{\frac{p}{q}}\left\|\left(A^*A\right)^{\frac{q-p}{2q}}M^{\frac{p-q}{2q}}\right\|_{\S_\infty^m},
\end{multline}

It follows from~\eqref{eq:unitary} that in order to prove the desired estimate~\eqref{eq:span upper AM} it suffice to establish that
\begin{equation}\label{eq:new goal}
\forall\, A\in X,\qquad \left\|\left(A^*A\right)^{\frac{q-p}{2q}}M^{\frac{p-q}{2q}}\right\|_{\S_\infty^m}\le \max\left\{k^{\frac{p-2}{2}\left(\frac{1}{p}-\frac{1}{q}\right)},1\right\} \|A\|_{\S_p^m}^{1-\frac{p}{q}}.
\end{equation}
To this end, suppose that $A\in X$ and apply Lemma~\ref{lem:beta}  with $\beta= \frac{q-p}{2q}\in (0,\frac12)$. It follows that
\begin{equation}\label{eq:set up for all p}
\left\|\left(A^*A\right)^{\frac{q-p}{2q}}M^{\frac{p-q}{2q}}\right\|_{\S_\infty^m}\le  \left(\trace\left[A^*A M^{\frac{p}{2}-1}\right]\right)^{\frac{q-p}{2q}}.
\end{equation}

Suppose first that $p\ge 2$. Then by~\eqref{eq:von neumann} with $S=A^*A$ and $T=M^{(p-2)/2}$, combined with  H\"older's inequality with exponents $p/2$ and $p/(p-2)$, the right hand side of~\eqref{eq:set up for all p} can be bounded as follows.
\begin{equation}\label{eq:for p>2}
\trace\left[A^*A M^{\frac{p}{2}-1}\right]\le \left\|A^*A\right\|_{\S_{\frac{p}{2}}}\left\|M^{\frac{p}{2}-1} \right\|_{\S_{\frac{p}{p-2}}}\stackrel{\eqref{eq:def schatten p}}{=}\|A\|_{\S_p}^2\left(\trace\left[M^{\frac{p}{2}}\right]\right)^{\frac{p-2}{p}}\stackrel{\eqref{eq:trace M identity}}{=} k^{1-\frac{2}{p}}\|A\|_{\S_p}^2.
\end{equation}
The case $p>2$ of the desired estimate~\eqref{eq:new goal} follows by substituting~\eqref{eq:for p>2} into~\eqref{eq:set up for all p}.

It remains to treat the range $1\le p< 2$ (which is the more substantial case; recall that for the present purposes, namely for proving Theorem~\ref{thm:markov convexity}, we need $p=1$).  Observe first that
\begin{eqnarray*}\label{eq:get to 2-p}
\nonumber \left\|\left(A^*A\right)^{\frac{q-p}{2q}}M^{\frac{p-q}{2q}}\right\|_{\S_\infty^m}&\stackrel{\eqref{eq:set up for all p}}{\le}&  \left(\trace\left[\left(A^*A\right)^{\frac{p}{2}} \left(A^*A\right)^{1-\frac{p}{2}}M^{\frac{p}{2}-1}\right]\right)^{\frac{q-p}{2q}}\\ \nonumber &\stackrel{\eqref{eq:ideal property}}{\le}& \left(\trace\left[\left(A^*A\right)^{\frac{p}{2}}\right]\right)^{\frac{q-p}{2q}}\left\|\left(A^*A\right)^{1-\frac{p}{2}}M^{\frac{p}{2}-1}\right\|_{\S_\infty^m}^{\frac{q-p}{2q}}
\\&\stackrel{\eqref{eq:def schatten p}}{=}&
\|A\|_{\S_p^m}^{\frac{p(q-p)}{2q}}\left\|\left(A^*A\right)^{1-\frac{p}{2}}M^{\frac{p}{2}-1}\right\|_{\S_\infty^m}^{\frac{q-p}{2q}}\nonumber \\&=&
\|A\|_{\S_p^m}^{1-\frac{p}{q}-\frac{(2-p)(q-p)}{2q}}\left\|\left(A^*A\right)^{1-\frac{p}{2}}M^{\frac{p}{2}-1}\right\|_{\S_\infty^m}^{\frac{q-p}{2q}}.
\end{eqnarray*}
Therefore, in order to establish the desired bound~\eqref{eq:new goal} when $1\le p< 2$, it suffices to prove that
\begin{equation}\label{eq:new goal 2-p}
\left\|\left(A^*A\right)^{1-\frac{p}{2}}M^{\frac{p}{2}-1}\right\|_{\S_\infty^m}\le \|A\|_{\S_p^m}^{2-p}.
\end{equation}
To this end, apply Lemma~\ref{lem:beta}  with $\beta= 1-\frac{p}{2}\in (0,\frac12]$; noting that this is only place in the proof of Lemma~\ref{lem:upper lewis} where the assumption $p\ge 1$ is used. It follows that
\begin{multline}\label{eq:bootstrap}
\left\|\left(A^*A\right)^{1-\frac{p}{2}}M^{\frac{p}{2}-1}\right\|_{\S_\infty^m}\stackrel{\eqref{eq:desired beta in lemma}}{\le} \left(\trace\left[A^*A M^{\frac{p}{2}-1}\right]\right)^{1-\frac{p}{2}}=\left(\trace\left[\left(A^*A\right)^{\frac{p}{2}} \left(A^*A\right)^{1-\frac{p}{2}}M^{\frac{p}{2}-1}\right]\right)^{1-\frac{p}{2}}\\\stackrel{\eqref{eq:ideal property}}{\le} \left(\trace\left[\left(A^*A\right)^{\frac{p}{2}}\right]\right)^{1-\frac{p}{2}}\left\|\left(A^*A\right)^{1-\frac{p}{2}}M^{\frac{p}{2}-1}\right\|_{\S_\infty^m}^{1-\frac{p}{2}}\stackrel{\eqref{eq:def schatten p}}{=}
\|A\|_{\S_p^m}^{p\left(1-\frac{p}{2}\right)}\left\|\left(A^*A\right)^{1-\frac{p}{2}}M^{\frac{p}{2}-1}\right\|_{\S_\infty^m}^{1-\frac{p}{2}}.
\end{multline}
By cancelling out the common terms in~\eqref{eq:bootstrap} and simplifying the resulting expression,  we arrive at~\eqref{eq:new goal 2-p}, thus completing the proof of Lemma~\ref{lem:upper lewis}.\end{proof}

The following theorem obtains the best-known upper bound on the distortion of an arbitrary finite-dimensional subspace of $\S_p$ in $\S_q$ for $q>p$. As we explain in Remark~\ref{rem:sharp} below, this bound is sharp when $1\le p<q\le 2$, and when $p>2$ obtaining the best possible bound in this context is an interesting (and challenging) open problem  (the corresponding question for $\ell_p$ is open as well).

\begin{theorem}\label{thm:embed into bigger q} Fix $p,q\in [1,\infty)$ with $p<q$. For every finite-dimensional linear subspace $X$ of $\S_p$,
\begin{equation}\label{cases distortion}
\cc_{\S_q}(X)\le \left\{\begin{array}{ll} \dim(X)^{\frac{1}{p}-\frac{1}{q}} &\mathrm{if}\ p\in [1,2],\\ \dim(X)^{\frac{p}{2}\left(\frac{1}{p}-\frac{1}{q}\right)}& \mathrm{if}\ p\in (2,\infty).\end{array}\right.
\end{equation}
\end{theorem}

\begin{proof} Denote $k=\dim(X)$. Fix $\e\in (0,1)$ and continue with the notation of Lemma~\ref{lem:discretization}, thus obtaining an embedding $\mathsf{J}:X\to \S_p^m$. Apply Lemma~\ref{lem:lewis} to the linear subspace $\mathsf{J}X$ of $\S_p^m$, thus obtaining  a basis $\{T_1,\ldots,T_k\}$ of $\mathsf{J}X$. Define a linear mapping $\Phi:X\to \S_q^m$ by setting
\begin{equation}\label{eq:define Phi}
\forall\, B\in X,\qquad \Phi B \eqdef (\mathsf{J} B)M^{\frac{q-p}{2q}},
\end{equation}
where $M\in \S_p^m$ is defined in~\eqref{eq:def M}. Then,
\begin{equation}\label{eq:Phi upper}
\|\Phi B\|_{\S_q^m}\stackrel{\eqref{eq:span upper AM}\wedge \eqref{eq:define Phi}}{\le} \max\left\{k^{\frac{p-2}{2}\left(\frac{1}{p}-\frac{1}{q}\right)},1\right\}  \|\mathsf{J} B\|_{\S_p^m}\stackrel{\eqref{eq:J isomorph}}{\le}\max\left\{k^{\frac{p-2}{2}\left(\frac{1}{p}-\frac{1}{q}\right)},1\right\}\|B\|_{\S_p}.
\end{equation}
For the reverse inequality,
\begin{equation}\label{eq:Phi lower}
(1-\e)\|B\|_{\S_p}\stackrel{\eqref{eq:J isomorph}}{\le} \left\|\mathsf{J}B\right\|_{\S_p^m}
\stackrel{\eqref{eq:A upper lewis}\wedge \eqref{eq:define Phi}}{\le} k^{\frac{1}{p}-\frac{1}{q}}\|\Phi B\|_{\S_q^m}.
\end{equation}
The desired distortion bound~\eqref{cases distortion} follows by combining~\eqref{eq:Phi upper} and~\eqref{eq:Phi lower} and letting $\e\to 0$.
\end{proof}

\begin{remark}\label{rem:sharp} When $1\le p<q\le 2$ we have $\cc_{\S_q}(\ell_p^k)=k^{\frac{1}{p}-\frac{1}{q}}$ for all $k\in \N$, i.e., the  bound of Theorem~\ref{thm:embed into bigger q} is attained for $X=\ell_p^k\subset \S_p$. Indeed, in~\cite{Dix53} it was shown  that the following estimate (the $\S_p$-version of the ``easy'' Clarkson inequality~\cite{Cla36}) holds true.
\begin{equation}\label{eq:clarkson}
\forall\, A,B\in \S_q,\qquad \frac{\|A+B\|_{\S_q}^q+\|A-B\|_{\S_q}^q}{2}\le \|A\|_{\S_q}^q+\|B\|_{\S_q}^q.
\end{equation}
We shall  now explain how~\eqref{eq:clarkson} implies that in fact $\cc_{\S_q}(\ell_p^k)\ge \cc_{\S_q}(\{-1,1\}^k,\|\cdot\|_{\ell_p^k})\ge  k^{\frac{1}{p}-\frac{1}{q}}$.

Arguing similarly to~\cite[Lemma~2.1]{LN04}, given $C_1,C_2,C_3,C_4\in \S_q$, apply~\eqref{eq:clarkson} twice, once with $A=C_2-C_1$ and $B=C_1-C_4$, and once with   $A=C_2-C_3$ and $B=C_3-C_4$. We obtain
\begin{equation*}
\frac{\|C_2-C_4\|_{\S_q}^q+\|C_2-2C_1+C_4\|_{\S_q}^q}{2}\le \|C_2-C_1\|_{\S_q}^q+\|C_1-C_4\|_{\S_q}^q,
\end{equation*}
and
\begin{equation*}
\frac{\|C_2-C_4\|_{\S_q}^q+\|C_2-2C_3+C_4\|_{\S_q}^q}{2}\le \|C_2-C_3\|_{\S_q}^q+\|C_3-C_4\|_{\S_q}^q.
\end{equation*}
By summing these inequalities and using convexity we conclude that
\begin{align}\label{eq:deduce roundness}
\nonumber \|C_1-C_2\|_{\S_q}^q+&\|C_2-C_3\|_{\S_q}^q+\|C_3-C_4\|_{\S_q}^q+\|C_4-C_1\|_{\S_q}^q\\ \nonumber&\ge
\|C_2-C_4\|_{\S_q}^q+\frac{\|C_2-2C_1+C_4\|_{\S_q}^q+\|2C_3-C_2-C_4\|_{\S_q}^q}{2}\\ \nonumber
&\ge \|C_2-C_4\|_{\S_q}^q+\left\|\frac12\big((C_2+C_4-2C_1)+(2C_3-C_2-C_4)\big)\right\|_{\S_q}^q \nonumber\\
&= \|C_1-C_3\|_{\S_q}^q+\|C_2-C_4\|_{\S_q}^q.
\end{align}
Using the terminology of~\cite{Enf69}, the estimate~\eqref{eq:deduce roundness} means that $\S_q$ has roundeness $q$. A well-known iterative application of roundness $q$ (see~\cite[Proposition~3]{Enf78} or~\cite[Proposition~5.2]{Oht09}), implies (using the terminology of~\cite{BMW86}) that $\S_q$ has Enflo type $q$, i.e., every $f:\{-1,1\}^n\to \S_q$ satisfies
\begin{equation}\label{eq:state enflo type}
\sum_{\e\in \{-1,1\}^k}\|f(\e)-f(-\e)\|_{\S_q}^q\le \sum_{j=1}^k\sum_{\e\in \{-1,1\}^k} \|f(\e)-f(\e_1,\ldots,\e_{j-1},-\e_j,\e_{j+1},\ldots,\e_k)\|_{\S_q}^q.
\end{equation}
So, if $\|\e-\e'\|_{\ell_p^k}\le \|f(\e)-f(\e')\|_{\S_q}\le\alpha\|\e-\e'\|_{\ell_p^k}$ for some $\alpha\in (0,\infty)$ and all $\e,\e'\in \{-1,1\}^k$, then by substituting these bounds into~\eqref{eq:state enflo type} we have $2^{k}(2k)^{q/p}\le k2^k(2\alpha)^q$. Thus necessarily $\alpha\ge k^{1/p-1/q}$.

When $2<p<q<\infty$ it remains an interesting open problem to determine the asymptotic behavior of the best possible upper bound in~\eqref{cases distortion}. The best known lower bounds in this context follow from the fact that $\S_q$ is an $X_q$ Banach space, as proved in~\cite{NS16}. If one wishes to obtain a discrete version of this lower bound as we did above, then the best known estimates follow from the $\S_q$-version of~\cite{Nao16} (which is not proved in~\cite{Nao16}  but is stated there: This statement was checked by the first named author in collaboration with A. Eskenazis and will appear elsewhere).
\end{remark}

\begin{proof}[Deduction of Theorem~\ref{thm:markov convexity} from Theorem~\ref{thm:embed into bigger q}] Following~\cite{Pis86}, the {\em martingale cotype $2$ constant} of a Banach space $(X,\|\cdot\|_X)$, denoted $\mathfrak{m}_2(X)$, is the infimum over those $\mathfrak{m}\in (0,\infty]$ such that for every probability space $(\mathscr{S},\mathscr{F},\mu)$, every martingale $\{M_k\}_{k=1}^\infty\subset L_2(\mu;X)$ satisfies
\begin{equation}\label{eq:def martingale cotype}
\sum_{k=1}^\infty \int_{\mathscr{S}}\|M_{k+1}-M_k\|_{X}^2\ud \mu \le \mathfrak{m}^2 \sup_{k\in \N} \int_{\mathscr{S}}\|M_k\|_{X}^2\ud \mu.
\end{equation}
Recalling the definition of the Markov 2-convexity constant $\Pi_2(X)$ of $X$ as presented in Section~\ref{sec:markov},  by combining~\cite{LNP06} and~\cite{MN08} we have $\mathfrak{m}_2(X)\asymp \Pi_2(X)$. Next, following~\cite{Bal92}, the $2$-convexity constant of a Banach space $X$, denoted $\mathscr{K}_2(X)$, is the infimum over those $K\in (0,\infty]$ such that
$$
\forall\, x,y\in X,\qquad 2\|x\|_X^2+\frac{2}{K^2}\|y\|_X^2\le \|x+y\|^2_X+\|x-y\|_X^2.
$$
It follows from~\cite{Pis75} that $\mathfrak{m}_2(X)\lesssim \mathscr{K}_2(X)$, and it was proved in~\cite{Bal92} that  the implicit constant in this inequality can be taken to be $1$, i.e., $\mathfrak{m}_2(X)\le \mathscr{K}_2(X)$. By~\cite{BCL94}, for $q\in (1,2]$ we have $\mathscr{K}_2(X)\le 1/\sqrt{q-1}$. Thus also $\mathfrak{m}_2(\S_q)\le 1/\sqrt{q-1}$ (in~\cite{Tom74} a weaker bound on $\mathfrak{m}_2(\S_q)$ was obtained that grows like $1/(q-1)^{O(1)}$ as $q\to 1$, which would suffice to derive an impossibility result for dimension reduction in $\S_1$ as in Theorem~\ref{thm:qs intro}, though with a worse dependence on $\alpha$ in the exponent).

By Theorem~\ref{thm:embed into bigger q}, for every finite-dimensional subspace $X$ of $\S_1$ and every  $q\in (1,2]$ we have
\begin{equation}\label{eq:to optimize q}
\Pi_2(X)\le \Pi_2(\S_q) \dim(X)^{1-\frac{1}{q}}\asymp \mathfrak{m}_2(\S_q)\dim(X)^{1-\frac{1}{q}}\le \frac{\dim(X)^{1-\frac{1}{q}}}{\sqrt{q-1}}.
\end{equation}
By optimizing over $q$ in~\eqref{eq:to optimize q} we thus complete the proof of Theorem~\ref{thm:markov convexity}.\end{proof}

\begin{remark}\label{rem:schauder} As promised earlier, we shall justify Lemma~\ref{lem:discretization} here. This amounts to a natural truncation argument which is included for the sake of completeness. Fix $\e\in (0,1)$ and let $\NN_{\e/2}$ be an $(\e/2)$-net in the unit sphere of $X$. For every $B\in \NN_{\e/2}$ let $\{f_j(B)\}_{j=1}^\infty, \{g_j(B)\}_{j=1}^\infty\subset \ell_2$ be  orthonormal eigenbases of $B^*B$ and $BB^*$, respectively, such that $B^*Bf_j(B)=\sigma_j(B)^2f_j(B)$ and $BB^*g_j(B)=\sigma_j(B)^2g_j(B)$. Also, fix $m(B)\in \N$ such that
\begin{equation}\label{eq:eps/3}
\bigg(\sum_{j=m(B)+1}^\infty \sigma_j(B)^p\bigg)\le \frac{\e}{4}.
\end{equation}
Consider the following linear subspaces of $\ell_2$
$$
F(B)\eqdef \spn\left(\{f_1(B),\ldots,f_{m(B)}(B)\}\right)\qquad\mathrm{and}\qquad G(B)\eqdef \spn\left(\{g_1(B),\ldots,g_{m(B)}(B)\}\right),
$$
and define
$$
V\eqdef \sum_{B\in \NN_{\frac{\e}{2}}}\big(F(B)+G(B)\big).
$$
Since $X$ is finite-dimensional, its unit sphere is compact and therefore $|\NN_{\e/2}|<\infty$. It follows that $V$ is a finite-dimensional subspace of $\ell_2$, so denote $m=\dim(V)\in \N$.

Observe that for every $B\in \NN_{\e/2}$, since $g_1(B),\ldots,g_{m(B)}(B)$ are eigenvectors of $BB^*$, the nonzero diagonal entries of the diagonalization of  $(\proj_{G(B)^\perp}B)(\proj_{G(B)^\perp}B)^*=\proj_{G(B)^\perp}BB^*\proj_{G(B)^\perp}$ coincide with the nonzero elements of $\{\sigma_j(B)^2\}_{j=m(B)+1}^\infty$. Hence,
\begin{equation}\label{proj GB}
\left\|\proj_{G(B)^\perp}B\right\|_{\S_p}=\bigg(\sum_{j=m(B)+1}^\infty\sigma_j(B)^p\bigg)^{\frac{1}{p}}\stackrel{\eqref{eq:eps/3}}{\le} \frac{\e}{4}.
\end{equation}
Since $G(B)\subset V$, we have $V^\perp\subset G(B)^\perp$, and therefore $\proj_{V^{\perp}}=\proj_{V^{\perp}}\proj_{G(B)^\perp}$. Consequently,
\begin{equation}\label{eq:multiplication on left proj}
\left\|\proj_{V^\perp}B\right\|_{\S_p}=\left\|\proj_{V^\perp}\proj_{G(B)^\perp}B\right\|_{\S_p} \stackrel{\eqref{eq:holder schatten}}{\le} \left\|\proj_{V^\perp}\right\|_{\S_\infty}\left\|\proj_{G(B)^\perp}B\right\|_{\S_p}  \stackrel{\eqref{proj GB}}{\le}\frac{\e}{4}.
\end{equation}
The analogous reasoning using the fact that $F(B)\subset V$ shows that also
\begin{equation}\label{eq:multiplication on right proj}
\left\|B\proj_{V^\perp}\right\|_{\S_p}\le \frac{\e}{4}.
\end{equation}
Hence, since $B-\proj_VB\proj_V=(\Id_{\ell_2}-\proj_V)B+\proj_VB(\Id_{\ell_2}-\proj_V)=\proj_{V^\perp}B+\proj_VB\proj_{V^\perp}$, where $\Id_{\ell_2}:\ell_2\to \ell_2$ is the identity mapping, by the triangle inequality in $\S_p$ we see that
\begin{multline}\label{eq:esp over 2 approx}
\left\|B-\proj_VB\proj_V\right\|_{\S_p}\le \left\|\proj_{V^\perp}B\right\|_{\S_p}+\left\|\proj_VB\proj_{V^\perp}\right\|_{\S_p}\\\stackrel{\eqref{eq:holder schatten}\wedge \eqref{eq:multiplication on left proj}}{\le} \frac{\e}{4}+\left\|\proj_V\right\|_{\S_\infty}\left\|B\proj_{V^\perp}\right\|_{\S_p}\stackrel{\eqref{eq:multiplication on right proj}}{\le}\frac{\e}{2}.
\end{multline}

Now, take any $A\in X$ with $\|A\|_{S_p}=1$. Observe that
\begin{equation}\label{eq:upper PAP}
\left\|\proj_VA\proj_V\right\|_{\S_p}\stackrel{\eqref{eq:holder schatten}}{\le} \left\|\proj_V\right\|_{\S_\infty}\left\|A\right\|_{\S_p}\left\|\proj_V\right\|_{\S_\infty}\le 1.
\end{equation}
Since $\NN_{\e/2}$ is an $(\e/2)$-net of the unit sphere of $X$, there is $B\in \NN_{\e/2}$ with $\|A-B\|_{\S_p}\le \e/2$. Then,
\begin{multline}\label{eq:lower PAP}
\left\|\proj_VA\proj_V\right\|_{\S_p}\ge \|B\|_{\S_p}-\left\|B-\proj_VB\proj_V\right\|_{\S_p}-\left\|\proj_V(A-B)\proj_V\right\|_{\S_p}\\\stackrel{\eqref{eq:holder schatten}\wedge \eqref{eq:esp over 2 approx}}{\ge}
1- \frac{\e}{2}-\left\|\proj_V\right\|_{\S_\infty}\left\|A-B\right\|_{\S_p}\left\|\proj_V\right\|_{\S_\infty}\ge 1-\e.
\end{multline}
Hence, if we fix an arbitrary isomorphism  $\mathscr{T}:\R^m\to V$ (recall that $V$ is $m$-dimensional) and define $\mathsf{J}A=\mathscr{T}^{-1}\proj_VA\proj_V\mathscr{T}:\R^m\to \R^m$, then the desired conclusion~\eqref{eq:J isomorph} follows from~\eqref{eq:upper PAP} and~\eqref{eq:lower PAP}.
\end{remark}


\bibliographystyle{abbrv}
\bibliography{nuclear-dimension-conf}

 \end{document}